\documentclass[12pt]{article}
\usepackage{amsgen, mathtools, amsmath, amsfonts, amsthm, amssymb, bbm, enumerate,amscd,graphicx,cite}
\usepackage{tikz-cd}
\usepackage[colorlinks=true,linkcolor=black,anchorcolor=black,citecolor=black,filecolor=black,menucolor=black,runcolor=black,urlcolor=black]{hyperref}
\usepackage{xcolor}
\newcommand{\sm}{\left(\begin{smallmatrix}}
\newcommand{\esm}{\end{smallmatrix}\right)}
\newcommand{\bpm}{\begin{pmatrix}}
\newcommand{\ebpm}{\end{pmatrix}}

\definecolor{blue}{rgb}{0,0,1}
\definecolor{red}{rgb}{1,0,0}
\definecolor{green}{rgb}{0,.6,.2}
\definecolor{purple}{rgb}{1,0,1}

\numberwithin{equation}{section}
\newtheorem{theorem}{Theorem}[section]
\newtheorem{lemma}[theorem]{Lemma}
\newtheorem{corollary}[theorem]{Corollary}

\newtheorem{proposition}[theorem]{Proposition}
\theoremstyle{definition}  
\newtheorem{definition} [theorem] {Definition}

\newtheorem{remark} [theorem] {Remark}

\newcommand{\keywords}[1]{\textbf{{Key Words:}} #1}
\newcommand{\NX}{(\Gamma_0(N), \chi)}
 
\long\def\red#1\endred{\textcolor{red}{#1}}
\long\def\blue#1\endblue{\textcolor{blue}{#1}}
\long\def\purple#1\endpurple{\textcolor{purple}{ #1}}
\long\def\green#1\endgreen{\textcolor{green}{#1}}

\title{Real-analytic modular forms for $\Gamma_0(N)$ and their $L$-series}
\author{Joshua Drewitt and Joshua Pimm}
\date{}

\begin{document}

\maketitle

\begin{abstract}

We extend the definition of real-analytic modular forms from $SL_2(\mathbb{Z})$ 
to congruence subgroups of the type $\Gamma_0(N)$. We examine their properties and discuss examples, such as real-analytic Eisenstein series and modular iterated integrals. We also associate an $L$-series to these forms and prove its functional equation. For the $L$-series of a special class of forms, which includes length-one modular iterated integrals, we establish a converse theorem.
\end{abstract}


\section{Introduction}

Brown \cite{brown1,brown2,brown3} introduced real-analytic modular forms, which share the characteristics of various previously studied important modular objects, such as holomorphic modular forms and harmonic Maass forms, but also have many of their own interesting features.  They have applications to the theory of mixed motives, single polylogarithms and to string perturbation theory. A subspace of particular importance to the latter is the space of modular iterated integrals $\mathcal{MI}^!$, which are currently being studied in detail  \cite{DrewittLap,DiamantisModit,Dorigoni:2022npe,Dorigoni:2023part1}.

Let $(r,s)\in\mathbb{Z}^2$. A real-analytic modular form $F$ of weights $(r,s)$ satisfies the transformation law
\begin{equation}
F(\gamma z)(cz+d)^{-r}(c\bar{z}+d)^{-s}=F(z),
\end{equation}
for all $z\in\mathbb{H}$ (the upper-half plane) and $\gamma=\begin{psmallmatrix}
        a & b \\ c & d
    \end{psmallmatrix}\in SL_2(\mathbb{Z})$. It also has an expansion of the form
\begin{equation}\label{generalexpintro}
F(z)=\sum_{|j| \le N_0}y^j\left ( \sum_{m, n \ge -N_0'} a_{m, n}^{(j)} \, q^m \overline{q}^n\right ) \!,
\end{equation}
where $N_0,N_0'\in\mathbb{N}$, $a_{m,n}^{(j)}\in\mathbb{C}$, and $q \coloneqq e^{2\pi i z}$. We denote the space of these forms by $\mathcal{M}^!_{r,s}$.

In \cite{nikoMe}, an $L$-series was defined which was appropriate for the study of length-one modular iterated integrals $\mathcal{MI}_1^!$, but it lacked certain features when compared to the theory of $L$-series associated to classical modular forms. In particular, their $L$-series was constructed to satisfy a functional equation, but not a converse theorem.

In recent work, Diamantis et. al. introduced $L$-series associated to harmonic Maass forms and established a functional equation and converse theorem \cite{diamantisLseries}. Harmonic Maass forms have been the subject of intense study in recent years, leading to many striking arithmetic, geometric and combinatorial results, but they are different modular objects to those studied here. Nevertheless, in this paper, we show that the construction of \cite{diamantisLseries} can be adapted to the setting of real-analytic modular forms. In particular, we establish a functional equation and converse theorem for a class of real-analytic modular forms that includes $\mathcal{MI}^!_1$.
For the purposes of the introduction, we now state weaker versions of our definitions and main theorems.
\begin{definition}
Let $F$ have an expansion of the form
\begin{equation}\label{Fexpformintro}
F(z)=\sum_{k\in S}y^k\sum_{n\geq n_0}a_n^{(k)}q^{n/M} + \sum_{k\in S}y^k\sum_{n\geq n_0}b_n^{(k)}\bar q^{n/M},
\end{equation}
for some $S\subset\mathbb{Z}$ a finite subset, $n_0\in\mathbb{Z}$, $M\in\mathbb{N}$. Let $\varphi\in\mathcal{F}_F$, where $\mathcal{F}_F$ is a space of test functions defined fully in Section \ref{Lseriessec}. The $L$-series associated to $F$ is then defined to be
\begin{equation}
L_F(\varphi) \coloneqq \sum_{k\in S}\sum_{n\geq n_0}a_{n}^{(k)} \,\mathcal{L}\varphi_{k+1}\left(\frac{2\pi n}{M}\right)+\sum_{k\in S}\sum_{n\geq n_0}b_{n}^{(k)} \,\mathcal{L}\varphi_{k+1}\left(\frac{2\pi n}{M}\right).
\end{equation}
where $\varphi_{k+1}(y) \coloneqq y^k\varphi(y)$, and $\mathcal{L}\varphi_{k+1}(s)$ is its Laplace transform.
\end{definition}

Finally, we write $\varphi|_{a}W_1(y) \coloneqq y^{-a}\varphi(1/y)$.
The above $L$-series satisfies the following functional equation and converse theorem:

\begin{theorem}[Functional equation]\label{funceqintro}
Let $F\in\mathcal{M}^!_{r,s}$ have an expansion of the form
\begin{equation}
F(z)=\sum_{k\in S}y^k\sum_{n\geq n_0}a_n^{(k)}q^{n} + \sum_{k\in S}y^k\sum_{n\geq n_0}b_n^{(k)}\bar q^{n}.
\end{equation}
Then, for $\varphi\in\mathcal{F}_F$ such that $\varphi|_{2-r-s} W_1\in\mathcal{F}_F$, we have the functional equation
\begin{equation}
L_{F}(\varphi)=i^{r-s}L_{F}(\varphi|_{2-r-s} W_1).
\end{equation}
\end{theorem}
\begin{theorem}[Converse theorem]\label{convintro}
Let $F$ have an expansion of the form
\begin{equation*}
F(z)=\sum_{k\in S}y^k\sum_{n\geq n_0}a_n^{(k)}q^{n} + \sum_{k\in S}y^k\sum_{n\geq n_0}b_n^{(k)}\bar q^{n},
\end{equation*}
where $a_n^{(k)},b_n^{(k)}\ll e^{Cn^{1/t}}$ for some $C$ positive, $t>1$. Assume that
\begin{align}
     L_{F}(\varphi) =  i^{r-s}L_{F}(\varphi|_{2-r-s} W_1), \label{Converse1aintro} 
     \\[2mm]
     L_{\partial_r F}(\varphi) =  -i^{r-s} L_{\partial_r F}(\varphi|_{2-r-s} W_1), \label{Converse1bintro} 
     \\[2mm]
     L_{\bar\partial_s F}(\varphi) =  -i^{r-s} L_{\bar\partial_s F}(\varphi|_{2-r-s} W_1), \label{Converse1cintro}
\end{align}
for some subset $X$ of $\mathcal{F}_F$ such that for all $y>0$, there exists a $\varphi\in X$ such that $\varphi(y)\neq0$. Here $\partial_r$ and $\overline\partial_s$ are Maass raising and lowering operators from \emph{\cite{brown1}}, see also \eqref{partial}. Then $F\in\mathcal{M}^!_{r,s}$.
\end{theorem}

The above two theorems are in analogy to the \emph{Hecke} functional equation and converse theorem. The main results in the paper, Theorem \ref{functionaleq} and Theorem \ref{converseth}, are in analogy with the theorems of \emph{Weil} (c.f. \cite{topics} chapter 7).

As far as the authors are aware, real-analytic modular forms have always been defined in relation to the  group $SL_2(\mathbb{Z})$ and, in particular, the theory has not yet been extended to include congruence subgroups. It is the introduction of real-analytic forms on $\Gamma_0(N)$ which allows us to exploit the full potential of the theory of $L$-series, going from the aforementioned Hecke theory to the more general Weil theory.

In this paper, we set up such a framework and provide a survey of real-analytic modular forms over $\Gamma_0(N)$:
\begin{equation*}
        \Gamma_0(N)  := \left\{ \begin{pmatrix}
        a & b \\ c & d
    \end{pmatrix}  \in SL_2(\mathbb{Z}) : c \equiv 0 \, {\rm mod} \, N \right\}.
\end{equation*}

After some preliminaries in Section \ref{Prelims}, we start in Section \ref{secprops} by defining such forms and studying how the differential operators $\partial$, $\overline{\partial}$ and the Laplacian $\Delta$ \eqref{lap} act on these forms. We look at examples of real-analytic modular forms and define the real-analytic Eisenstein series $\mathcal{E}_{u}(z,r,s,\chi)$ associated to a cusp $u$ and character $\chi$. In Section \ref{Lseriessec}, we define the $L$-series associated to real-analytic modular forms, and prove the functional equation (Theorem \ref{functionaleq}) and converse theorem (Theorem \ref{converseth}).

This theory was designed with the space of special interest $\mathcal{MI}^!_1$ in mind. However, it is possible to define an $L$-series for general real-analytic modular forms, not just those satisfying the expansion type (\ref{Fexpformintro}), but the more general (\ref{generalexpintro}). For the more general $L$-series, it comes at the expense of an additional so-called `auxiliary variable' $u$, originally implemented by Bykowski \cite{bykovskii2000functional} and used in the context of double Dirichlet series and metaplectic Eisenstein series by Diamantis-Goldfeld in \cite{DiamantisGoldfeld2}. We can also state and prove a functional equation and converse theorem in this case, and we sketch the method in Section \ref{Lseriessec}.

\vspace{4mm}

\textbf{Acknowledgements.} We thank N. Diamantis for his valuable support and feedback during this writing process and M. Lee for her helpful comments and encouragement. The research of JD was supported by the Royal Society (Spectral theory of automorphic forms: trace formulas and more). JP was supported by EPSRC DTP grant EP/R513283/1.

\newpage

\section{Preliminaries}\label{Prelims}
In this section, we set up the basic groundwork needed in order to extend the definition of real-analytic modular forms from $SL_2(\mathbb{Z})$ to certain congruence subgroups. We begin by introducing such subgroups:

Let $N\geq1$ be an integer, we are interested in (level $N$) congruence subgroups that are of the form
\begin{equation}
\Gamma_0(N)\coloneqq \left\{ \begin{pmatrix}
        a & b \\ c & d
    \end{pmatrix}  \in SL_2(\mathbb{Z}) : c \equiv 0 \, {\rm mod} \, N \right\}.
\end{equation}
For a cusp $u$ of $\Gamma_0(N)$, the stabiliser of $u$ is given by the group
\begin{equation*}
\Gamma_u \coloneqq \{  \gamma \in \Gamma : \gamma u = u \}.
\end{equation*}
There exists a $\sigma_u \in SL_2(\mathbb{Z})$, called a scaling matrix of $u$, such that $\sigma_u (\infty)= u$ and
\begin{equation}
    \sigma_u^{-1} \Gamma_u \sigma_u = \Gamma_\infty = \left\{ \pm \begin{pmatrix}
1 &  b \\
0 &  1  
\end{pmatrix}  : b \in \mathbb{Z} \right\}.
\end{equation}
We can see that $\gamma_u \coloneqq \sigma_u^{-1} \begin{psmallmatrix}
    1 & 1\\
    0 & 1
\end{psmallmatrix} \sigma_u$ and $-\gamma_u$ generate the stabiliser $\Gamma_u$. To make computations simpler throughout, we will always take
\begin{equation}
  \sigma_\infty = \begin{pmatrix}
    1 & 0 \\
    0 & 1
\end{pmatrix}.
\end{equation}
Let $\chi$ be a Dirichlet character modulo $N$, then, for $\gamma = \begin{psmallmatrix}
a &  b \\
c &  d 
\end{psmallmatrix}  \in SL_2(\mathbb{Z})$ , we define $\chi(\gamma)=\chi(d)$ (and note that $\overline{\chi(d)}= \chi(a)$). A cusp $u$ is said to be singular for $\chi$ if $\chi$ is trivial on $\Gamma_u$. 
We note that a Dirichlet character $\chi$ is always trivial on  $\Gamma_\infty$ and, therefore, the cusp $\infty$ is singular for any $\chi$.

We set $r,s$ to be integers throughout and, for $\gamma = \begin{psmallmatrix}
a &  b \\
c &  d 
\end{psmallmatrix} \in GL^+(2,\mathbb{R})=\{\gamma\in GL(2,\mathbb{R}):\det(\gamma)>0\}$, we define $j(\gamma, z) \coloneqq (cz+d)$. For a function $f\colon \mathbb{H} \to \mathbb{C}$, the double slash operator $f||_{r,s}$ is then defined by
\begin{equation*}
    (f ||_{r,s}\gamma)(z) \coloneqq \det(\gamma)^{\frac{r+s}{2}}j(\gamma, z)^{-r}j(\gamma, \bar z)^{-s}f(\gamma z), \quad \  \forall \ z \in \mathbb{H}.
\end{equation*}
This is a group action, and therefore, we have \begin{equation*}
f||_{r,s} (\gamma \mu) = (f||_{r,s} \gamma)||_{r,s} \, \mu \, ,\quad \forall \ \gamma,  \mu \in GL^+(2,\mathbb{R}).
\end{equation*}

\section{Properties of real-analytic modular forms for \texorpdfstring{$\Gamma_0(N)$}{Gamma(N)}}\label{secprops}

\subsection{Key definitions}\label{secdef}

We are now able to study the space of real-analytic modular forms for $\Gamma_0(N)$. We start by giving the definition of such forms, and then spend the rest of the section studying their fundamental properties and characteristics.
Unless stated otherwise, we will always assume $\chi$ is a Dirichlet character modulo $N$. 

\begin{definition}\label{ramfdef}
    Let $N\in \mathbb{N}$ and $\chi$ be a character. We say a real-analytic function $F: \mathbb{H} \to \mathbb{C}$ is a real-analytic modular form of weights $(r,s)$ for $\Gamma_0(N)$ with character $\chi$ if it satisfies the following three conditions: 
    \begin{align}
(i)\quad \  & F(\gamma z)=\chi(d)j(\gamma,z)^r j(\gamma, \bar z)^s F(z), \quad \ \forall \ \gamma \in \Gamma_0(N) \text{ and } z \in \mathbb{H}. \label{transform}
\\[1mm]
(ii)\quad \  &  \text{At the cusp $\infty$ we have an expansion of the form} \nonumber \\
& \hspace{15mm} F(z)=\sum_{|j| \le N_0}y^j\left ( \sum_{m, n \ge -N_0'} a_{m, n}^{(j)} \, q^m \overline{q}^n\right ) \!,  \label{fourier}\\&
\text{where $N_0,N_0'\in\mathbb{N}$, $a_{m,n}^{(j)}\in\mathbb{C}$, and $q \coloneqq e^{2\pi i z}$.} \nonumber\\[1mm]
(iii)\quad \ &  \text{For all $\gamma\in SL_2(\mathbb{Z})$ there is a $C>0$ such that}\nonumber \\
& \hspace{15mm} (F||_{r,s}\gamma)(x+iy)=O(e^{Cy})\\&
\text{as $y\rightarrow\infty$, uniformly in $x$. We will refer to this as the growth of $F$ \emph{at the cusps}}. \nonumber
\\ \nonumber
    \end{align}
\end{definition}

We denote the space of such real-analytic modular forms by $\mathcal{M}^!_{r,s}(\Gamma_0(N), \chi)$ and set $\mathcal{M}^!(\Gamma_0(N), \chi) \! =\! \bigoplus_{r,s}\mathcal{M}^!_{r,s}(\Gamma_0(N), \chi)$. The notation $\mathcal{M}^!_{r,s}(\Gamma_0(N))$ and $\mathcal{M}^!(\Gamma_0(N))$ will be used throughout to refer to those spaces where $\chi$ is trivial. Since $\begin{psmallmatrix}
-1 &  0 \\
0 &  -1  
\end{psmallmatrix} \in \Gamma_0(N)$, we will assume that $\chi(-1)=(-1)^{r+s}$.

If we restrict to the condition that the $a^{(j)}_{m,n}$ vanish if either $m$ or $n$ are negative, then we have a new (sub)space denoted by $\mathcal{M}_{r,s}\NX$.

Let $A$ be a complete set of inequivalent cusps for $\Gamma_0(N)$, if the expansion of $F(z) \in \mathcal{M}^!_{r,s}(\Gamma_0(N), \chi)$ at any singular cusp $u \in A$ can be given by
\begin{equation}\label{fourierext}
   F(\sigma_u z)j(\sigma_u, z)^{-r}j(\sigma_u, \bar z)^{-s}=\sum_{|j| \le N_0}y^j\left ( \sum_{m, n \ge -N_0'} a_{m, n}^{u,j} \,   q^m \overline{q}^n \right ) \!,
\end{equation}
 for some $N_0,N_0' \in \mathbb{N}$ and $a_{m, n}^{u,j} \in \mathbb{C}$, then 
 condition $(iii)$ of Definition \ref{ramfdef} is implied. The above expansion reduces to \eqref{fourier} when taking $\sigma_u = \sigma_\infty = \begin{psmallmatrix}
    1 & 0 \\
    0 & 1
\end{psmallmatrix}$ and, as we can assume that $A$ includes the cusp at $\infty$, condition $(ii)$ is also implied. We will refer to these expansions as being `at the cusps'.

As with the classical case of $\mathcal{M}^!_{r,s}(SL_2(\mathbb{Z}))$, we equip the space $\mathcal{M}^!_{r,s}(\Gamma_0(N), \chi)$ with a pair of operators given by 
\begin{equation}\label{partial}
    \partial_r = 2i \, \text{Im}(z) \dfrac{\partial}{\partial z} +r \qquad \text{and} \qquad \overline{\partial}_s = -2i \, \text{Im}(z) \dfrac{\partial}{\partial \bar{z}} +s.
\end{equation}
We have the following two lemmas from \cite{brown1} (and an additional corollary):

\begin{lemma}[Lemma 2.5 of \cite{brown1}]\label{oplem1}
For all $\gamma \in SL_2(\mathbb{R})$, $z \in \mathbb{H}$ and real analytic $F: \mathbb{H} \to \mathbb{C}$, we have
\begin{align*}
  &  \partial_r \left(j(\gamma,z)^{-r}F(\gamma z)\right) = 
\, j(\gamma,z)^{-r-1}j(\gamma, \bar{z})(\partial_r F)(\gamma z), \\
  &  \overline{\partial}_s \left(j(\gamma,\bar{z})^{-s}F(\gamma z)\right) = \, j(\gamma,z)j(\gamma, \bar{z})^{-s-1}(\overline{\partial}_s F)(\gamma z).
\end{align*}
\end{lemma}

\begin{corollary}\label{opcor}
    For all $\gamma \in SL_2(\mathbb{R})$, $z \in \mathbb{H}$ and real analytic $F: \mathbb{H} \to \mathbb{C}$, we have
    \begin{align*}
        \partial_r ((F||_{r,s} \gamma)(z)) = (( \partial_r F)||_{r+1,s-1} \gamma )(z),
        \\
        \overline{\partial}_s ((F||_{r,s} \gamma)(z)) = (( \overline{\partial}_s F)||_{r-1,s+1} \gamma )(z).
    \end{align*}
\end{corollary}

\begin{lemma}[Lemma 2.6 of \cite{brown1}]\label{oplem2}
    The operators $\partial_r, \overline{\partial}_s$ preserve the expansion appearing in the RHS of equation \eqref{fourier}. Their action is given explicitly by
    \begin{align*}
        & \partial_r(y^j q^m \bar{q}^n ) = (-4\pi m y+r+j) y^jq^m\bar{q}^n, \\
        & \overline{\partial}_s(y^j q^m \bar{q}^n ) = (-4\pi n y+s+j) y^jq^m\bar{q}^n.
    \end{align*}
\end{lemma}

Therefore, $\partial_r$ and $\overline{\partial}_s$ induce bigraded operators $ \partial, \overline{\partial}$ on real-analytic functions $F:\mathbb{H}\rightarrow\mathbb{C}$ of weight $(r,s)$ of bidegree $(1,-1)$ and $(-1,1)$, respectively. If $F$ has weight $(r,s)$, then $\partial$ acts on $F$ via $\partial_r$ and similarly $\overline{\partial}$ acts via $\overline{\partial}_s$. Unfortunately, condition $(iii)$ of Definition $\ref{ramfdef}$ is not in general preserved by the $ \partial,  \overline{\partial}$ operators.

We give the following proposition relating to the kernel of the $\partial$ operator, which is an extension of Proposition 3.2 of\cite{brown1}:

\begin{proposition}\label{unique}
    If $F \in \mathcal{M}_{r,s}\NX$ such that $\partial_r F = 0$, then 
    \begin{equation*}
        y^r F \in \overline{M}_{s-r}\NX,
    \end{equation*}
    where $\overline{M}_{s-r}\NX$ denotes the space of anti-holomorphic forms of weight $s-r$ for $\Gamma_0(N)$ with character $\chi$. If $r > s$ then $F$ vanishes, and if $r=s$ then $F \in \mathbb{C} y^{-r}$.
\end{proposition}
\begin{proof}
    The same proof used for Proposition 3.2 of\cite{brown1} holds here since there are no non-zero holomorphic forms of negative weight for $\Gamma_0(N)$ and character $\chi$. The final part is due to the fact that that $M_0\NX = \mathbb{C}$.
\end{proof}

This proposition means that the solution to $\partial F = G$, for $F,G \in \mathcal{M}_{r,s}\NX$, is unique if $r>s$ and unique up to an addition of $g \in \mathbb{C}y^{-r}$ if $r=s$. We can obtain an analogous result for $\overline{\partial}$.

 We can use a combination of the $\partial_r$ and $\overline{\partial}_s$ operators to construct the Laplacian $\Delta_{r,s}$. This is defined for all integers $r,s$ by
\begin{equation}\label{lap}
\Delta_{r,s} = -\overline{\partial}_{s-1}\partial_r + r(s-1) = - \partial_{r-1} \overline{\partial}_s + s(r-1).
\end{equation}
which induces a bigraded operator $\Delta$ of bidegree $(0,0)$ on real-analytic functions of weight $(r,s)$.

Now that we have introduced real-analytic modular forms for $\Gamma_0(N)$ and the associated operators $\partial$, $\overline{\partial}$  and $\Delta$, we are in a good position to study the real-analytic Eisenstein Series.



\subsection{Real-analytic Eisenstein series}

The classical holomorphic Eisenstein series for $SL_2(\mathbb{Z})$, defined for all even $r \geq 4$ by
\begin{equation*}
    \mathbb{G}(z,r) := \dfrac{(r-1)!}{2(2\pi i)^{r}}  \sum_{\substack{(m,n)\in \mathbb{Z}^2 \\ (m,n)\neq (0,0)}} (mz+n)^{-r} = \dfrac{\zeta(r)(r-1)!}{(2\pi i)^{r}} \sum_{\gamma \in \Gamma_\infty \backslash SL_2(\mathbb{Z})} j(\gamma, z)^{-r},
\end{equation*}
 naturally gives rise to a more general definition when studying the congruence subgroup $\Gamma_0(N)$. Indeed, we can define the holomorphic Eisenstein series associated to a singular cusp $u$ of $\Gamma_0(N)$ and a character $\chi$ by
\begin{equation}\label{Eish}
    \mathbb{G}_{u}(z,r,\chi) \coloneqq \dfrac{\zeta(r)(r-1)!}{(2\pi i)^{r}} \sum_{\gamma \in \Gamma_u \backslash \Gamma_0(N)} \dfrac{ \overline{\chi}(\gamma)}{j(\sigma^{-1}_u\gamma, z)^{r}}.
 \end{equation}
This series is defined for any integer $r \geq 3$ provided that $j(\sigma_u \gamma \sigma^{-1}_u, z)^r=1$ for all $\gamma \in \Gamma_u$, when $r$ is even this is guaranteed to hold.

We also have the non-holomorphic Eisenstein series associated to a singular cusp $u$ and character $\chi$, defined for all $\text{Re}(r)>1$ by
\begin{equation}\label{Eisnh}
E_u(z,r,\chi) \coloneqq \zeta(2r)\sum_{\gamma \in \Gamma_u \backslash \Gamma_0(N)} \overline{\chi}(\gamma)(\text{Im}(\sigma^{-1}_u \gamma z))^r.
\end{equation}

Both of these objects have already undergone intense study in previous literature and have a variety of applications (an interested reader could see \cite{topics,miyakebook,iwaniecspectral, cohen1,goldsteindedekind,selberg1956,liuwang}, for example, to learn more). Our reason for introducing them here, however, is to see how they relate to the real-analytic Eisenstein series for $\Gamma_0(N)$, which we will define below.
Indeed, as in the classical case above, we wish to generalise the real-analytic Eisenstein series for $SL_2(\mathbb{Z})$, defined by Brown in \cite{brown1} and given by
\begin{equation}\label{Eisb}
\mathcal{E}(z,r,s) \coloneqq \dfrac{w!}{(2\pi i)^{w+1}} \dfrac{1}{2}\sum_{\substack{(m,n)\in \mathbb{Z}^2 \\ (m,n)\neq (0,0)}}  \dfrac{i\, \text{Im}(z)}{(mz+n)^{r+1}(m\bar z + n)^{s+1}},
\end{equation}
for $r,s \geq 0$ and $w \coloneqq r+s \geq 2$ and even, to the congruence subgroup $\Gamma_0(N)$. As a way of unifying all of the above Eisenstein series, we will also show how our function relates to \eqref{Eisnh} and satisfies differential equations involving \eqref{Eish}.

\begin{definition}\label{Eisdef}
    Fix a congruence subgroup $\Gamma_0(N)$ and let $r,s \geq 0$ be integers such that $w \coloneqq r+s \geq 2$ and $j(\sigma_u \gamma \sigma^{-1}_u, z)^{w
    }=1$. Then we define the real-analytic Eisenstein series associated to a singular cusp $u$ and a character $\chi$ by
    \begin{equation}\label{EisA}
        \mathcal{E}_{u}(z,r,s,\chi)   \coloneqq \dfrac{w! \, \zeta(w+2)}{(2\pi i)^{w+1}} \sum_{\gamma \in \Gamma_u \backslash \Gamma_0(N)} \dfrac{ \overline{\chi}(\gamma) \: \! i \: \! \text{Im}(z) }{j(\sigma^{-1}_u \gamma, z)^{r+1}j(\sigma^{-1}_u \gamma, \bar{z})^{s+1}}.
    \end{equation}
\end{definition}

\begin{proposition}\label{eisprop}
Fix a congruence subgroup $\Gamma_0(N)$, a character $\chi$ and a singular cusp $u$ of $\Gamma_0(N)$. For every $w \geq 2$ such that $j(\sigma_u \gamma \sigma^{-1}_u, z)^w=1$, there exists a unique set of functions in $\mathcal{M}_{r,s}\NX$, with $r,s \geq 0$ and $r+s=w$, which satisfies the following equations:
\begin{align}
    \partial \mathcal{E}_{u}(z,w,0,\chi) =-2\pi
    \emph{Im}(z)\mathbb{G}_{u}(z,w+2,\chi)&, \label{ediff1}
    \\[1mm]
    \partial \mathcal{E}_{u}(z,r,s,\chi) - (r+1)\mathcal{E}_{u}(z,r+1,s-1,\chi)=0&,\ \hspace{10mm} \text{for } s\geq 1. \label{ediff2}
    \\[4mm]
    \overline{\partial} \mathcal{E}_u (z,0,w, \chi) =-2\pi \emph{Im}(z)\overline{\mathbb{G}}_{u} (z,w+2,\chi)&, \label{ediff3}
    \\[1mm]
    \overline{\partial}  \mathcal{E}_{u}(z,r,s,\chi) - (s+1)\mathcal{E}_u (z,r-1,s+1,\chi)=0&,\ \hspace{10mm} \text{for } r\geq 1. \label{ediff4}
\end{align}
This set of functions is given exactly by the above Eisenstein series $\mathcal{E}_{u}(z,r,s,\chi)$.
\end{proposition}
\begin{proof}
    This proof mostly follows the same proof that was given in the $SL_2(\mathbb{Z})$ case (see Proposition 4.1 of \cite{brown1}):

\begin{itemize}
    \item Uniqueness comes from Proposition \ref{unique} and the remarks following it.
    \item Equation \eqref{EisA} converges and obeys equation \eqref{transform}.   
    \item We have
    \begin{align*}
E_u(z,r+1,\chi) & = \zeta(2r+2) \! \sum_{\gamma \in \Gamma_u \backslash \Gamma_0(N)} \! \overline{\chi}(\gamma)(\text{Im}(\sigma^{-1}_u \gamma z))^{r+1} 
\\
& =  \zeta(2r+2) \! \sum_{\gamma \in \Gamma_u \backslash \Gamma_0(N)} \dfrac{\overline{\chi}(\gamma) \text{Im}(z)^{r+1}}{j(\sigma^{-1}_u \gamma, z)^{r+1}j(\sigma^{-1}_u \gamma, \bar{z})^{r+1}}
    \end{align*}
    and hence
    \begin{equation*}
\mathcal{E}_{u}(z,r,r,\chi) = \dfrac{i}{y^r}\dfrac{(2r)!}{(2\pi i)^{2r+1}} E_u(z,r+1,\chi).
    \end{equation*}
The expansion of the right-hand side is known to lie in $\mathbb{C}[[q,\bar{q}]][y^{\pm}]$, the expansions of the RHS at the other cusps are also known to lie in $\mathbb{C}[[q,\bar{q}]][y^{\pm}]$ (see \cite{topics}, for example). The expansions of the $\mathcal{E}_{u}(z,r,s,\chi)$ at the cusps are deduced from the known expansions for $\mathcal{E}_{u}(z,r,r,\chi)$, by applying $\partial, \bar{\partial}$. Therefore, equation \eqref{EisA} satisfies \eqref{fourierext}.

\item The two above points prove that $\mathcal{E}_{u}(z,r,s,\chi)$ is a real-analytic modular form of weights $(r,s)$ for $\Gamma_0(N)$ with character $\chi$. It remains to show that equations \eqref{ediff1} to \eqref{ediff4} hold.
\item Equation \eqref{ediff2} holds by the following identity (for $r,s\in \mathbb{Z}$):
    \begin{equation*}
        \partial_r \left( \dfrac{z-\bar{z}}{(cz+d)^{r+1}(c\bar{z}+d)^{s+1}} \right) = (r+1)\left(\dfrac{z-\bar{z}}{(cz+d)^{r+2}(c\bar{z}+d)^s} \right)
    \end{equation*}
    and equation \eqref{ediff4} holds by complex conjugation.
      \item Equation \eqref{ediff1} follows from the above point and the identity below:
      \begin{equation*}
           -2\pi \text{Im}(z)\mathbb{G}_{u}(z,w+2,\chi) = \dfrac{\zeta(w+2)(w+1)!}{(2\pi i)^{w+1}} \sum_{\gamma \in \Gamma_u \backslash \Gamma_0(N)} \dfrac{ \overline{\chi}(\gamma) \: \! i \: \! \text{Im}(z) } {j(\sigma^{-1}_u\gamma, z)^{w+2}}
      \end{equation*}
      and \eqref{ediff3} follows by complex conjugation. \qedhere
\end{itemize}
\end{proof}

We note that we retrieve the Eisenstein series given by Brown when $\Gamma_0(N) = \Gamma \coloneqq SL_2(\mathbb{Z})$, $\chi$ is trivial and $r+s$ is even (here the only cusp is $u=\infty$):
\begin{equation*}
    \mathcal{E}(z,r,s) = \dfrac{w!}{(2\pi i)^{w+1}} \dfrac{1}{2}\sum_{\substack{(m,n)\in \mathbb{Z}^2 \\ (m,n)\neq (0,0)}}  \dfrac{i \, \text{Im}(z)}{(mz+n)^{r+1}(m\bar z + n)^{s+1}}
\end{equation*}
and writing a pair $(m,n)\in \mathbb{Z}^2 \neq (0,0)$ as $(gm',gn')$, where $g=gcd(m,n)$ and $m',n'$ are coprime, gives
\begin{align*}
   \mathcal{E}(z,r,s)  & =  \dfrac{w!}{(2\pi i)^{w+1}}  \dfrac{1}{2}\sum^{\infty}_{g=1}\sum_{\substack{(m',n')\in \mathbb{Z}^2 \\ gcd(m',n')=1}}  \dfrac{i \, \text{Im}(z)}{g^{r+1}(m'z+n')^{r+1}g^{s+1}(m'\bar z + n')^{s+1}} 
    \\ 
    & = \dfrac{w!}{(2\pi i)^{w+1}} \dfrac{1}{2}\sum^{\infty}_{g=1} g^{w+2} \sum_{\substack{(m',n')\in \mathbb{Z}^2 \\ gcd(m',n')=1}}  \dfrac{i \, \text{Im}(z)}{(m'z+n')^{r+1}(m'\bar z + n')^{s+1}} 
    \\ & = \dfrac{w! \, \zeta(w+2)}{(2\pi i)^{w+1}}  \sum_{\gamma \in \Gamma_{\infty} \backslash \Gamma}  \dfrac{i \, \text{Im}(z)}{j(\gamma, z)^{r+1} j(\gamma, \bar{z})^{s+1}} 
\end{align*}
then, since $\sigma^{-1}_{\infty} = 1$, we have
\begin{align*}
     \mathcal{E}(z,r,s) & = \dfrac{w! \, \zeta(w+2)}{(2\pi i)^{w+1}} \sum_{\gamma \in \Gamma_{\infty} \backslash \Gamma}  \dfrac{i \, \text{Im}(z)}{j(\sigma^{-1}_{\infty}\gamma, z)^{r+1} j(\sigma^{-1}_{\infty}\gamma, \bar{z})^{s+1}} \\ 
     & = \mathcal{E}_\infty (z,r,s,1).
\end{align*}

This new Eisenstein series is also an eigenfunction of the Laplacian, this is summarised in the following corollary:

\begin{corollary}
The real-analytic Eisenstein series (associated to a singular cusp $u$ and a character $\chi$) is an eigenfunction of the Laplacian with eigenvalue $-w=-r-s$:
\begin{equation*}
    \Delta   \mathcal{E}_{u}(z,r,s,\chi)  = -w   \mathcal{E}_{u}(z,r,s,\chi). 
\end{equation*}
\end{corollary}
\begin{proof}
    This follows from equations \eqref{lap}, \eqref{ediff2} and  \eqref{ediff4}. 
\end{proof}

We will now study the general theory of eigenfunctions of $\Delta$.

\subsection{Eigenfunctions of the Laplacian}\label{seceig}

For $\lambda \in \mathbb{C}$, we let $
\mathcal{HM}^!(\Gamma_0(N),\chi ; \lambda) \coloneqq \mathrm{Ker} (\Delta-\lambda : \mathcal{M}^!\NX \to \mathcal{M}^!\NX)$ denote the space of eigenfunctions of $\Delta$ that have eigenvalue $\lambda$. The following lemma captures some of the key characteristics shared by all of the functions in this space:
\begin{lemma}[Lemma 5.2 of \cite{brown1}]\label{eigenexp}
    Let $F \in \mathcal{HM}^!(\Gamma_0(N),\chi ; \lambda) $ be of weights $(r,s)$, then the eigenvalue $\lambda$ is an integer (or $F=0$):
    \begin{equation*}
        \mathcal{HM}^!\NX = \bigoplus_{n \in \mathbb{Z}} \mathcal{HM}^!(\Gamma_0(N),\chi ; n).
    \end{equation*}
    Furthermore, there exists an integer $k_0$ such that $\lambda = - k_0(k_0 + r+s - 1)$. If $k_0 \neq 1 -r-s-k_0$, we let $k'_0 = \mathrm{min}\{ k_0, 1-r-s-k_0 \} $ and there is a unique decomposition of $F= F^h + F^a + F^0 $ given by
    \begin{align}
F^h(z) \coloneqq \sum^{-s}_{j=k'_0} y^j \sum_{\substack{m \geq -N \\m \neq 0}} a_m^{(j)} q^m, \label{fh}\\
F^a(z) \coloneqq \sum^{-r}_{j=k'_0} y^j \sum_{\substack{m \geq -N' \\m \neq 0}} b_m^{(j)} \bar{q}^m, \label{fa}\\
F^0(z) \coloneqq ay^{k'_0} + by^{1-r-s-k'_0}, \nonumber
\end{align}
for some $a, b, a_m^{(j)},b_m^{(j)} \in \mathbb{C}$ and $N, N'  \in \mathbb{N}$. 
If $k_0 = 1-r-s-k_0$, then we have a unique decomposition similar to the above but with $F^0 = ay^{k'_0}$.  

Finally $F^h$, $F^a$, $y^{k'_0}$ and $y^{1-r-s-k'_0}$ are eigenfunctions of $\Delta$ with eigenvalue $\lambda$.
\end{lemma}
\begin{proof}
    This proof is essentially identical to the proof of the $\mathcal{HM}(SL_2(\mathbb{Z}))$ case (see Lemma 5.2 of \cite{brown1}), which also holds for $\mathcal{HM}^!(SL_2(\mathbb{Z}))$ (see Lemma 4.3 of \cite{brown3} and the remarks following it). The main difference being that in the classical case $k_0$ is always distinct from $1-r-s-k_0$, whereas here $k_0$ may equal $1-r-s-k_0$. This is due to the fact that $r+s \in 2\mathbb{Z}$ for $\mathcal{M}_{r,s}^!(SL_2(\mathbb{Z}))$, which is not necessarily true for $\mathcal{M}^!_{r,s}\NX$ (see Lemma 5.2 of \cite{brown1} for more details). This difference leads to the addition of the second decomposition.
\end{proof}

For all $\gamma\in SL_2(\mathbb{R})$, the action of the double slash operator $||_{r,s}\gamma$ preserves eigenfunctions of $\Delta_{r,s}$.

\begin{lemma}\label{efinv} Suppose $F$ is real analytic with $\Delta_{r,s} F = \lambda F$, for some eigenvalue $\lambda$, and let $\gamma\in SL_2(\mathbb{R})$. Then
\begin{equation}
\Delta_{r,s}(F||_{r,s}\gamma)=\lambda(F||_{r,s}\gamma).
\end{equation}
\end{lemma}
\begin{proof} This is a consequence of equation \eqref{lap} and Corollary \ref{opcor}.
\end{proof}

\subsection{Modular iterated integrals}

We fix a congruence subgroup $\Gamma_0(N)$ and a character $\chi$, and then define $\mathcal{E}(\Gamma_0(N), \chi)$ to be the space of all real-analytic Eisenstein series associated to $\chi$ and any singular cusp of $\Gamma_0(N)$. By Proposition \ref{eisprop}, this space is almost, but not quite, closed under the action of the operators $\partial$ and $\overline{\partial}$:
\begin{align}
    \partial \mathcal{E}(\Gamma_0(N), \chi) \subset \mathcal{E}(\Gamma_0(N), \chi) + \mathbb{G}(\Gamma_0(N),\chi)[y], \label{diffEis}
    \\
    \overline{\partial} \mathcal{E}(\Gamma_0(N), \chi) \subset \mathcal{E}(\Gamma_0(N), \chi) + \overline{\mathbb{G}(\Gamma_0(N),\chi)}[y], \nonumber
\end{align}
where $\mathbb{G}(\Gamma_0(N),\chi)$ is the space of all holomorphic Eisenstein series associated to $\chi$ and any singular cusp of $\Gamma_0(N)$.


We can generalise this idea to define a space of functions in $\mathcal{M}^!_{r,s}(\Gamma_0(N), \chi)$, called the space of modular iterated integrals, which will be  closed under the operators $\partial$ and $\overline{\partial}$,
up to the addition of an extra component (and its complex conjugation). This extra component will include the space $M^!(\Gamma_0(N), \chi)$ of weakly holomorphic modular forms, which, in turn, includes $\mathbb{G}(\Gamma_0(N),\chi)$. We first recall the definition of $M^!(\Gamma_0(N), \chi)$ and then define the space of modular iterated integrals.

A variation of the transformation equation used to define real-analytic modular forms \eqref{transform} is given by
\begin{equation*}
    f(\gamma z)=\chi(d)j(\gamma,z)^r f(z), \quad \ \forall \ \gamma \in \Gamma_0(N) \text{ and } z \in \mathbb{H}.
\end{equation*}
A holomorphic function on $\mathbb{H}$ that satisfies this equation and is meromorphic at the cusps of $\Gamma_0(N)$ is a weakly holomorphic modular form of weight $r$ for $\Gamma_0(N)$ with character $\chi$. The space of such forms is denoted by $M^!_r(\Gamma_0(N), \chi)$ and we set $M^!(\Gamma_0(N), \chi) = \bigoplus_{r} M^!_r(\Gamma_0(N), \chi)$.

\begin{definition}
    We let $\mathcal{MI}^!_{-1}(\Gamma_0(N),\chi)=0$. For any integer $k \geq 0$  we define the space of modular iterated integrals of length $k$, $\mathcal{MI}^!_{k}(\Gamma_0(N),\chi)$, to be the largest subspace of $\bigoplus_{r,s \geq 0} \mathcal{M}^!_{r,s}(\Gamma_0(N),\chi)$ which satisfies 
\begin{align}
    \partial \mathcal{MI}^!_k(\Gamma, \chi) &\subset \mathcal{MI}^!_k(\Gamma,\chi) + M^!(\Gamma, \chi)[y] \otimes \mathcal{MI}^!_{k-1}(\Gamma, \chi), \label{MIeq} \\[1mm]
    \overline{\partial} \mathcal{MI}^!_k(\Gamma,\chi) &\subset \mathcal{MI}^!_k(\Gamma,\chi) +      \overline {M^!(\Gamma, \chi)}[y] \otimes \mathcal{MI}^!_{k-1}(\Gamma,\chi), \nonumber
\end{align}
where we have set $\Gamma \coloneqq \Gamma_0(N)$ to keep the above expressions more succinct.

\end{definition}

This space is closed under complex conjugation and is an extension of the space given by Brown in Definition 5.1 of \cite{brown3}, we also note that our definition reduces to Brown's one when $\Gamma_0(N)= SL_2(\mathbb{Z})$ and $\chi$ is trivial.

The simplest example we can look at, for non-negative $k$, is when $k=0$. We begin by introducing the following two lemmas (these lemmas are extensions of Lemma 3.1 and Corollary 3.5 from \cite{brown3}, their proofs work in a very similar manner and so they are not included for the sake of brevity): 

\begin{lemma}\label{lemmaker}
For all weights $r,s$ we have the following:
\begin{align*}
        & \hspace{3mm} \mathcal{M}^!_{r,s}(\Gamma_0(N), \chi) \cap \emph{ker}  (\partial_r)  \subset y^{-r} \mathbb{C} [[\overline{q}]],
        \\[1mm]
        & \hspace{3mm} \mathcal{M}^!_{r,s}(\Gamma_0(N), \chi) \cap \emph{ker}  (\overline\partial_s)   \subset y^{-s} \mathbb{C} [[q]],
        \\[2mm]
        &\mathcal{M}^!(\Gamma_0(N), \chi) \cap \emph{ker}  (\partial) \cap \emph{ker}  (\overline\partial)  \subset  \mathbb{C}[y^\pm].
\end{align*}
\end{lemma}

\begin{lemma}\label{lemmadf=0}
    Let $F \in \mathcal{M}^!_{r,s}(\Gamma_0(N), \chi)$ with $h:=r-s \geq 0$, and suppose that
    \begin{equation*}
        \partial F = 0 \qquad \text{and}\qquad \overline{\partial}^{h+1}F = 0.
    \end{equation*}
    Then $F$ vanishes unless $h=0$, in which case $F \in \mathbb{C} y^{-r}$.
\end{lemma}
    
Using the above two lemmas we can now give a characterisation of $\mathcal{MI}^!_0(\Gamma_0(N), \chi)$.

\begin{proposition}
Fix a congruence subgroup $\Gamma_0(N)$ and a character $\chi$, then the associated space of length-zero modular iterated integrals is given by
\begin{equation}\label{MI0}
    \mathcal{MI}^!_0(\Gamma_0(N), \chi)= \mathbb{C}[y^{-1}].
\end{equation}
\end{proposition}
\begin{proof}
    The proof for Proposition 5.2 of \cite{brown3} also works in this case when it is combined with Lemmas \ref{lemmaker} and \ref{lemmadf=0} above.
\end{proof}

Using this proposition, we can then view $\mathcal{MI}^!_1(\Gamma_0(N), \chi)$ as the largest subspace of $\bigoplus_{r,s \geq 0}$ $ \mathcal{M}^!_{r,s}(\Gamma_0(N),\chi)$ which satisfies the following: 
\begin{align*}
        \partial \mathcal{MI}^!_1(\Gamma_0(N), \chi) &\subset \mathcal{MI}^!_1(\Gamma_0(N),\chi) + M^!(\Gamma_0(N), \chi)[y^{\pm}] , \\[1mm]
    \overline{\partial} \mathcal{MI}^!_1(\Gamma_0(N),\chi) &\subset \mathcal{MI}^!_1(\Gamma_0(N),\chi) +      \overline {M^!(\Gamma_0(N), \chi)}[y^{\pm}],
\end{align*}
which is more reminiscent of \eqref{diffEis}. We can deduce from the above that the real-analytic Eisenstein series $\mathcal{E}_{u}(z,r,s,\chi)$ is a length-one modular iterated integral.
We are also able to extend Proposition 5.8 of \cite{brown3} to our more general class of modular iterated integrals, this allows us to decompose length-one modular iterated integrals as sums of eigenfunctions of the Laplaican. 

\begin{proposition}
    Any length-one modular iterated integral $F \in \mathcal{MI}^!_1(\Gamma_0(N), \chi)$ of weights $(r,s)$ can be uniquely decomposed as a linear combination
    \begin{equation*}
        F = \sum_{0 \leq k \leq \rm{min}\{ r,s\} } F_k,
    \end{equation*}
    where the $F_k$ also exist in $\mathcal{MI}^!_1(\Gamma_0(N), \chi)$  and have modular weights $(r,s)$. Furthermore, they are eigenfunctions of the Laplacian with eigenvalue $(k-1)(r+s-k)$:
    \begin{equation*}
        \left(\Delta - (k-1)(r+s-k) \right) F_k = 0.
    \end{equation*}
\end{proposition}

\section{L-series for real-analytic modular forms}\label{Lseriessec}

Fix $r,s\in\mathbb{Z}$. We start by defining the matrix
\begin{equation*}
W_N=
\begin{pmatrix}
    0 & -N^{-1/2} \\
    N^{1/2} & 0 
\end{pmatrix} \!,
\end{equation*}
which will be useful throughout this section, along with the following lemma:

\begin{lemma}\label{lemG|F}
Let $F \in \mathcal{M}^!_{r,s}(\Gamma_0(N), \chi)$ and set $G \coloneqq F||_{r,s} W_N$, then
\begin{equation*}
    G||_{r,s} \gamma = \overline{\chi}(d) G, \qquad \forall \ \gamma = \begin{psmallmatrix}
        a & b \\
        c & d
    \end{psmallmatrix} \in \Gamma_0(N)
\end{equation*}
\end{lemma}
\begin{proof}
    If $\gamma \in \Gamma_0(N)$ is of the form $\begin{psmallmatrix}
a &  b \\
c &  d  
\end{psmallmatrix} $, then $ c/N \in \mathbb{Z}$ and
\begin{equation*}
W_N \gamma W_N^{-1} = \begin{pmatrix}
    d & -c/N \\
    - bN & a
\end{pmatrix} \in \Gamma_0(N).
\end{equation*}
Therefore, we see that
\begin{equation*}
G||_{r,s} \gamma = F||_{r,s} W_N ||_{r,s} \gamma = F||_{r,s} W_N \gamma W_N^{-1}||_{r,s} W_N = \overline{\chi}(d) F||_{r,s} W_N = \overline{\chi}(d) G,
\end{equation*}
where we used the fact that, since $N|c$, $\overline\chi(d) = \chi(a)$. 
\end{proof}

We also introduce the single slash operator, this is defined for a function $\varphi\colon \mathbb{R_+} \to \mathbb{C}$ and an integer $a$ by 
\begin{equation*}
    (\varphi|_a W_N)(x)\coloneqq (Nx)^{-a}\varphi\left(\dfrac{1}{Nx}\right), \qquad \forall \ x \in \mathbb{R_+}.
\end{equation*}
Now, for each smooth function $\varphi:\mathbb{R}_+ \to \mathbb{C}$, we define the Laplace transform by
\begin{equation*}
    (\mathcal{L}\varphi)(s) = \int^{\infty}_{0} e^{-st} \varphi(t) dt,
\end{equation*}
for $s \in \mathbb{C}$ such that the integral converges absolutely. We also define
\begin{equation*}
    \varphi_s(x) \coloneqq \varphi(x) x^{s-1}
\end{equation*}
 where, for example, $  \varphi_1(x) = \varphi(x) $. We denote the space of piece-wise smooth complex-valued functions on $\mathbb{R}$ by $C(\mathbb{R}, \mathbb{C})$. Now let $M \in \mathbb{Z}_{>0}$ and suppose $F$ is a function on $\mathbb{H}$ given by the series 
  
\begin{equation}\label{F/M}
F(z) = \sum_{k\in S} \sum_{n\geq n_0} y^k a_n^{(k)} q^{n/M} + \sum_{k\in S} \sum_{n\geq n_0} y^k b_n^{(k)} \bar{q}^{n/M}
\end{equation}
where $n_0\in\mathbb{Z}$ and $S\subset\mathbb{Z}$ is a finite subset. We then let $\mathcal{F}_F$ be the space of functions $\varphi \in  C(\mathbb{R}, \mathbb{C})$ such that the integral defining $(\mathcal{L}\varphi)(s)$ converges absolutely for all $s$ with ${\rm Re}(s) \geq 2\pi n_0$ and we have the following convergence for all $k\in S$:
\begin{equation}\label{Ffconv}
\sum_{k\in S}\sum_{n\geq n_0}|a_{n}^{(k)}|\mathcal{L}|\varphi_{k+1}|\left(\frac{2\pi n}{M}\right)+\sum_{k\in S}\sum_{n\geq n_0}|b_{n}^{(k)}|\mathcal{L}|\varphi_{k+1}|\left(\frac{2\pi n}{M}\right)<\infty.
 \end{equation}

\begin{definition}
    Let $M \in \mathbb{Z}_{>0}$ and suppose $F$ is a function on $\mathbb{H}$ given by the expansion \eqref{F/M}. We define the $L$-series associated to $F$ to be the map  $L_F: \mathcal{F}_F \to \mathbb{C}$, for $\varphi \in \mathcal{F}_F$, given by
    \begin{align*}
        &L_F(\varphi)=\sum_{k\in S}\sum_{n\geq n_0}a_{n}^{(k)}\mathcal{L}\varphi_{k+1}\left(\frac{2\pi n}{M}\right)+\sum_{k\in S}\sum_{n\geq n_0}b_{n}^{(k)}\mathcal{L}\varphi_{k+1}\left(\frac{2\pi n}{M}\right)\!.
    \end{align*}
\end{definition}

We also have the following integral expression for the $L$-series:

\begin{lemma}\label{integralformlemma}
When $F$ is of the form $(\ref{F/M})$ and $\varphi\in\mathcal{F}_F$, then
\begin{equation}\label{integralform}
L_F(\varphi)=\int^{\infty}_0 F(iy)\varphi(y) dy.
\end{equation}
\end{lemma}
\begin{proof}
Write out the expansion of $F$ inside the integral and use the absolute convergence condition relating to  $\varphi$ to exchange sums and integrals. $\square$
\end{proof}

Now, suppose $F$ is a function on $\mathbb{H}$ given by the expansion \eqref{F/M} with $M=1$. Furthermore, let $D$ be a positive integer and $\chi$ a Dirichlet character modulo $D$. We then define the twisted function $F_{\chi}$ by
\begin{align}\label{Fchidef}
    F_{\chi}(z) \coloneqq & D^{(r+s)/2} \sum_{\mu \, \rm{mod} \, D} \overline{\chi(\mu)} \left(F\Big|\Big|_{r,s} \! \begin{pmatrix}
D^{-1/2} &  \mu \, D^{-1/2} \\
0 &  D^{1/2}
\end{pmatrix} \! \right)(z) .
\end{align}
This definition can also be written as
\begin{align*}
F_\chi (z) &= \sum_{\mu \, \rm{mod} \, D} \overline{\chi(\mu)}  \,f\left(\frac{z + \mu}{D}\right)
\\
&= \sum_{\mu \, \rm{mod} \, D} \overline{\chi(\mu)} 
 \left[ \sum_{k\in S} \left(\frac{y}{D}\right)^k \sum_{n\geq n_0} a_n^{(k)} e^{2\pi i n (z+\mu)/D} + \sum_{k\in S} \left(\frac{y}{D}\right)^k \sum_{n\geq n_0}  b_n^{(k)} e^{-2\pi i n (\bar z +\mu)/D}\right] .
\end{align*}

We can rewrite this equation further by introducing the generalised Gauss sum, which, for a character $\chi$ and integer $n$, is given by
\begin{equation*}
\tau_{\chi}(n) \coloneqq \sum_{\mu \, \rm{mod} \, D} \chi(\mu) e^{\frac{2\pi i n \mu}{D} }.
\end{equation*}
Using this expression, we then have 
\begin{equation*}
    F_\chi (z) =  \sum_{k\in S} \sum_{n\geq n_0} \left(\frac{y}{D}\right)^k \tau_{\bar\chi}(n) a_n^{(k)} q^{n/D} 
 + \sum_{k\in S} \sum_{n\geq n_0} \left(\frac{y}{D}\right)^k \overline{\tau_\chi (n)} b_n^{(k)} \bar q^{n/D}  
\end{equation*}
and, therefore, we can easily see that an expansion of the form given by \eqref{F/M} is preserved by twists of Dirichlet functions.
We will prove below that eigenfunctions of $\Delta$ are also preserved by such twists.

\begin{lemma}
Suppose $F$ is real analytic and $\Delta_{r,s} F=\lambda F$ for some eigenvalue $\lambda$. Then $\Delta_{r,s}F_\chi=\lambda F_\chi$ for any Dirichlet character $\chi$.
\end{lemma}
\begin{proof} By Lemma \ref{efinv}, each summand in the expression (\ref{Fchidef}) of $F_\chi$ is an eigenfunction of $\Delta_{r,s}$ with eigenvalue $\lambda$. And since $F_\chi$ is a finite sum of such eigenfunctions, it is itself an eigenfunction.
\end{proof}

\begin{theorem}\label{functionaleq}
Let $r,s \in \mathbb{Z}$, $N,D \in \mathbb{N}$ with $(D,N)=1$, and let $\psi$ and $\chi$ be Dirichlet characters modulo $N$ and $D$, respectively. Furthermore, let $F\in\mathcal{M}^!_{r,s}(\Gamma_0(N), \psi)$ have an expansion of the form
\begin{equation}
F(z) = \sum_{k\in S} \sum_{n\geq n_0} y^k a_n^{(k)} q^n + \sum_{k\in S} \sum_{n\geq n_0} y^k b_n^{(k)} \bar{q}^n,
\end{equation}
with coefficients not growing too rapidly e.g.
\begin{equation}\label{coeffgrowth}
a_n^{(k)},b_n^{(k)}\ll e^{Cn^{1/t}},
\end{equation}
for all $k\in S$, and some constants $C>0$ and $t>1$.
Suppose a similar expansion and coefficient bounds for
\begin{align*}
G \coloneqq F||_{r,s} W_N.\\
\end{align*}
Finally, we define the test function space
\begin{align*}
\mathcal{F}_{F,G} \coloneqq \bigcap_{\chi\mod D}\left\{\varphi\in\mathcal{F}_{F_\chi}:\varphi|_{2-r-s} W_N\in\mathcal{F}_{G_{\bar\chi}}\right\}.
\end{align*}
Then $\mathcal{F}_{F,G}$ is non-empty and, for $\varphi \in \mathcal{F}_{F,G}$, we have the following functional equation:
\begin{align}
     L_{F_\chi}(\varphi) =  \frac{i^{r-s}\chi(-N)\psi(D)}{N^{\frac{r+s-2}{2}}} L_{G_{\bar\chi}}(\varphi|_{2-r-s} W_N) \label{LFX1b} .
\end{align}

\end{theorem}
\begin{proof}
Let $S_c(\mathbb{R_+})$ be a set of complex-valued, compactly supported, piecewise continuous functions $\varphi:\mathbb{R_+}\rightarrow\mathbb{C}$ such that for all $y>0$, there exists a $\varphi\in S_c(\mathbb{R_+})$ with the property that $\varphi (y)\neq 0$. If $\varphi$ is in $S_c(\mathbb{R_+})$ then so is $\varphi|_{2-r-s}W_N$. Therefore, it suffices to show that $S_c(\mathbb{R_+})\subset\mathcal{F}_{F_\chi}\cap\mathcal{F}_{G_{\bar\chi}}$ for all $\chi$ mod $D$. For $\varphi\in S_c(\mathbb{R_+})$, suppose Supp$(\varphi)\subset(c_1,c_2)$, $c_1,c_2>0$ and $|\varphi(y)|\leq C_\varphi$ for some $C_\varphi>0$ for all $y\in$ Supp$(\varphi)$. Then
\begin{align*}
    \mathcal{L}|\varphi_{k+1}|\left(\frac{2\pi n}{D}\right) & = 
    \int^{\infty}_0e^{\frac{-2\pi ny}{D}}|\varphi_{k+1}(y)|dy
    \\[2mm]
    & = \int^{c_2}_{c_1}e^{\frac{-2\pi ny}{D}}|y^k\varphi(y)|dy
    \\[2mm]
    & \leq C_\varphi(c_2-c_1)c_2^ke^{\frac{-2\pi n c_1}{D}}.
\end{align*}
Given the condition (\ref{coeffgrowth}) on the $a_{n}^{(k)},b_{n}^{(k)}$, and similar coefficients for $G$, and since $\tau_\chi(n)=O_D(1)$, the series defining $L_{F_{\chi}}(\varphi)$ and $L_{G_{\bar\chi}}(\varphi)$ converges for all $\chi\mod D$. And we deduce that $\varphi\in\mathcal{F}_{F,G}$.

Recalling the definition of $F_\chi$, we have
\begin{align*}
F_{\chi}||_{r,s} W_N & =  D^{(r+s)/2} \sum_{\mu \, \rm{mod} \, D} \overline{\chi(\mu)} \left(F\Big|\Big|_{r,s} \! \begin{pmatrix}
D^{-1/2} &  \mu \, D^{-1/2} \\
0 &  D^{1/2}
\end{pmatrix}  
W_N
\right)
\\
& = D^{(r+s)/2} \sum_{\mu \, \rm{mod} \, D} \overline{\chi(\mu)} \left(F\Big|\Big|_{r,s} \! W_N \, \Big|\Big|_{r,s} W_N^{-1}
\begin{pmatrix}
D^{-1/2} &  \mu \, D^{-1/2} \\
0 &  D^{1/2}
\end{pmatrix}  
W_N
\right).
\end{align*}
Using $G\coloneqq F||_{r,s} W_N$ and the identity
\begin{equation*}
W_N^{-1}
\begin{pmatrix}
D^{-1/2} &  \mu \, D^{-1/2} \\
0 &  D^{1/2}
\end{pmatrix}  
W_N =  \begin{pmatrix}
D &  - \nu \\
-N\mu &  (1+N\mu\nu)D^{-1}
\end{pmatrix}  
\begin{pmatrix}
D^{-1/2} &   \nu D^{-1/2} \\
0 &  D^{1/2}
\end{pmatrix},
\end{equation*}
for integers $\mu$ and $\nu$ such that ${\rm gcd}(\mu,D)=1$ and $N\mu\nu \equiv -1 \, {\rm mod} \, D$, the above becomes
\begin{align*}
 F_{\chi}||_{r,s} W_N &=  D^{(r+s)/2} \sum_{\mu \, \rm{mod} \, D} \overline{\chi(\mu)} \left(G\Big|\Big|_{r,s} \begin{pmatrix}
D &  - \nu \\
-N\mu &  (1+N\mu\nu)D^{-1}
\end{pmatrix}  
\begin{pmatrix}
D^{-1/2} &   \nu D^{-1/2} \\
0 &  D^{1/2}
\end{pmatrix} \!
\right)
\\
& =  \chi(-N) \psi(D) D^{(r+s)/2} \sum_{\nu \, \rm{mod} \, D} \chi(\nu) \left(G\Big|\Big|_{r,s}
\begin{pmatrix}
D^{-1/2} &   \nu D^{-1/2} \\
0 &  D^{1/2}
\end{pmatrix}\!
\right) \!.
\end{align*}
For the final equality we used Lemma \ref{lemG|F} and the fact that, due to the properties of the $\mu,\nu$ we have chosen, $\overline{\chi}(\mu)= \chi(-N)\chi(\nu)$. This allows us to deduce the equality
\begin{equation}\label{FX=GX}
F_\chi ||_{r,s} W_N = X(-N)\psi(D) G_{\bar{\chi}}.
\end{equation}
Making the change of variable $y \to 1/(Ny)$ in equation \eqref{integralform} and then using $(F_\chi||_{r,s} W_N) (iy) = i^{s-r}N^{(-r-s)/2}F_\chi (-1/Niy) y^{-r-s}$ gives
\begin{align*}
    L_{F_\chi}(\varphi) & = \dfrac{1}{N} \int^{\infty}_0 F_\chi \left(\frac{i}{Ny}\right) \, \varphi\left(\dfrac{1}{Ny}\right)  y^{-2} dy 
    \\[2mm]
    & = \dfrac{1}{N} \int^{\infty}_0 F_\chi \left(\frac{-1}{Niy}\right) \, \varphi\left(\dfrac{1}{Ny}\right)  y^{-2} dy 
    \\[2mm]
   & = \dfrac{i^{r-s}}{N^{\frac{-r-s+2}{2}}}\int^{\infty}_0  (F_\chi||_{r,s} W_N)(iy) \,   \varphi\left(\dfrac{1}{Ny}\right)   y^{r+s-2} dy 
    \\[2mm]
  & =  \dfrac{i^{r-s} X(-N)\psi(D)}{N^{\frac{-r-s+2}{2}}} \int^{\infty}_0   G_{\bar{\chi}}(iy) \,   \varphi\left(\dfrac{1}{Ny}\right)   y^{r+s-2} dy,
\end{align*}
where we used the equation \eqref{FX=GX} for the last equality.
Finally, using $(\varphi|_{2-r-s} W_N)(y) = (Ny)^{r+s-2} \varphi\big(1/(Ny)\big)$ we have
\begin{align*}
     L_{F_\chi}(\varphi) & = \dfrac{i^{r-s} X(-N)\psi(D)}{N^{\frac{r+s-2}{2}}}\int^{\infty}_0   G_{\bar{\chi}}\left(iy\right) \,   \big(\varphi |_{2-r-s}W_N\big)\!\big(y\big) \, dy \\
     & = \dfrac{i^{r-s} X(-N)\psi(D)}{N^{\frac{r+s-2}{2}}} \, L_{G_{\bar\chi}}(\varphi|_{2-r-s} W_N),
\end{align*}
which gives equation \eqref{LFX1b} as required.
\end{proof}

\subsection{Converse}

To prove our converse theorem, we use the following lemma, in analogy to \cite{Bum} (Lemma 1.9.2):

\begin{lemma}\label{BumLem}
Suppose $F$ has an expansion of the form
\begin{equation}\label{fexpform}
F(z)=\sum_{k\in S}y^k\sum_{n\geq n_0}a_n^{(k)}q^{n/M} + \sum_{k\in S}y^k\sum_{n\geq n_0}b_n^{(k)}\bar q^{n/M},
\end{equation}
for $S\subset\mathbb{Z}$ a finite subset and $n_0\in\mathbb{Z}$, and with the $a_n^{(k)},b_n^{(k)}\ll e^{Cn^{1/t}}$ as in \emph{(\ref{coeffgrowth})}. Further suppose that
\begin{equation}
F(iy)=\frac{\partial}{\partial x}F(iy)=0,
\end{equation}
for all $y>0$. Then $F(z)=0$ for all $z\in\mathbb{H}$.
\end{lemma}
\begin{proof}
We have that
\begin{equation}
F(iy)=\sum_{n\geq n_0}P_n(y)e^{-2\pi n y/M}
\end{equation}
for some `polynomials' with negative powers
\begin{equation}\label{Pn}
P_n(y)=\sum_{k\in S}(a_n^{(k)}+b_n^{(k)})y^k.
\end{equation}
Similarly, we have
\begin{equation}
\frac{\partial}{\partial x}F(iy)=\sum_{n\geq n_0}Q_n(y)e^{-2\pi n y/M},
\end{equation}
where
\begin{equation}\label{Qn}
Q_n(y)=\frac{2\pi i n}{M}\sum_{k\in S}(a_n^{(k)}-b_n^{(k)})y^k.
\end{equation}
Now we consider
\begin{equation}\label{Fiypolyform}
F(iy)=\sum_{n<0}P_n(y)e^{-2\pi n y/M}+\sum_{n\geq 0}P_n(y)e^{-2\pi n y/M}.
\end{equation}
If nonzero, the right left-hand term grows exponentially in $y$ and the right-hand term is $O(y^{k'})$ for $k'$ the largest element of $S$, by the condition on the $a_n^{(k)},b_n^{(k)}$. This is a contradiction, and therefore we at least have the left-hand term being uniformly zero. That is,
\begin{equation}\label{exptermgone}
F(iy)=\sum_{n\geq 0}P_n(y)e^{-2\pi n y/M}=0.
\end{equation}
But then we can successively multiply equation (\ref{exptermgone}) by $e^{2\pi n y/M}$ to gain exponentially growing terms like those on the left side of (\ref{Fiypolyform}), which by the same reasoning must vanish. Inductively, then, $P_n(y)=0$ for all $n\geq n_0$. In a similar way, we can get $Q_n(y)=0$ for all $n\geq n_0$. Comparing with (\ref{Pn}) and (\ref{Qn}) gives us a pair of simultaneous equations with the unique solution $a_n^{(k)}=b_n^{(k)}=0$ for all $n\geq n_0$. That is to say, $F\equiv 0.$
\end{proof}

Note: If $F$ is of the form (\ref{fexpform}) with $M=1$, then, for $\chi$ a character modulo $D$, $F_\chi$ is also of the form (\ref{fexpform}) with $M=D$.

\begin{remark}
By Lemma \ref{eigenexp}, functions in $\mathcal{M}^!(\Gamma_0(N), \psi)$ of the form given by \eqref{fexpform} include $\mathcal{HM}^!(\Gamma_0(N),\psi)$, and linear combinations including $\mathcal{MI}_1(\Gamma_0(N),\psi)$. Hence our converse theorem can account for quite a broad class of functions.
\end{remark}

For $F$ of the form \eqref{fexpform}, we have the following explicit expansions for $(\partial_r F)(z)$ and $(\overline{\partial}_s F)(z)$ (and note that these are also of that form):
\begin{align*}
    (\partial_r F)(z) &= rF(z) + \sum_{k\in S} M^{-k} \sum_{n\geq n_0} y^k \left[k-\frac{4\pi n y}{M} \right] a_n^{(k)} q^{n/M} + \sum_{k\in S} M^{-k} \sum_{n\geq n_0}  ky^k  b_n^{(k)} \bar q^{n/M} ,
    \\
    (\overline{\partial}_s F)(z) &= sF(z) + \sum_{k\in S} M^{-k} \sum_{n\geq n_0} ky^k a_n^{(k)} q^{n/M} + \sum_{k\in S} M^{-k} \sum_{n\geq n_0}  y^k \left[k-\dfrac{4\pi n y}{M} \right]  b_n^{(k)}  \bar q ^{n/M}.
\end{align*}

By Lemma \ref{integralformlemma}, we have
\begin{align}
     L_{\partial_r (F_{\chi})}(\varphi) = \int^{\infty}_0 \partial_r (F_{\chi})(iy)\varphi(y) dy,
     \\[2mm]
     L_{\bar\partial_s (F_{\chi})}(\varphi) = \int^{\infty}_0 \overline{\partial}_s (F_{\chi})(iy)\varphi(y) dy,
\end{align}
for $\varphi\in\mathcal{F}_{\partial_r (F_{\chi})}\cap\mathcal{F}_{\bar\partial_s (F_{\chi})}$. We will use these $L$-series in the converse theorem below.

\begin{theorem}\label{converseth}
Let $N,r,s$ be integers with $N \geq 1$ and $\psi$ be a Dirichlet character modulo $N$. Suppose $a^{(k)}_n,b^{(k)}_n,A^{(k)}_n, B^{(k)}_n$ are complex numbers with a growth condition like
\begin{equation}\label{coeffgrowth2}
a_n^{(k)},b_n^{(k)},A_n^{(k)},B_n^{(k)}\ll e^{Cn^{1/t}},
\end{equation}
for some $C$ positive, $t>1$. We consider the expansions

 \begin{equation}\label{Conversecusp}
    F_1 = \sum_{k\in S} \sum_{n\geq n_0} y^k a_n^{(k)} q^m + \sum_{k\in S} \sum_{n\geq n_0} y^k b_n^{(k)} \bar{q}^n
 \end{equation}
 and
 \begin{equation}
     F_2 = \sum_{k\in S} \sum_{n\geq n_0} y^k A_n^{(k)} q^m + \sum_{k\in S} \sum_{n\geq n_0} y^k B_n^{(k)} \bar{q}^n.
 \end{equation}
  
Finally, assume that for any $\varphi\in S_c(\mathbb{R_+})$, $D\in \{1,2,...,N^2-1 \}$ with \emph{gcd}$(D,N)=1$ and Dirichlet character $\chi$ mod $D$, the following equations hold:
\begin{align}
     L_{F_{1_\chi}}(\varphi) =  \frac{\chi(-N)\psi(D)}{i^{s-r} N^{\frac{r+s-2}{2}}} L_{F_{2_{\bar\chi}}}(\varphi|_{2-r-s} W_N), \label{Converse1a} 
     \\[2mm]
     L_{\partial_r F_{1_\chi}}(\varphi) =  -\frac{\chi(-N)\psi(D)}{i^{s-r} N^{\frac{r+s-2}{2}}} L_{\partial_r F_{2_{\bar\chi}}}(\varphi|_{2-r-s} W_N), \label{Converse1b} 
     \\[2mm]
     L_{\bar\partial_s F_{1_\chi}}(\varphi) =  -\frac{\chi(-N)\psi(D)}{i^{s-r} N^{\frac{r+s-2}{2}}} L_{\bar\partial_s F_{2_{\bar\chi}}}(\varphi|_{2-r-s} W_N). \label{Converse1c}
\end{align}
Then, $(F_{1_{\chi}}||_{r,s}W_N) (z) = \chi(-N)\psi(D)   F_{2_{\bar\chi}}(z)$ and $F_1 \in \mathcal{M}^!_{r,s}(\Gamma_0(N), \psi)$.
\end{theorem}
\begin{proof}
Let $\varphi\in S_c(\mathbb{R_+})$ with Supp$(\varphi)\subset[c_1,c_2]$. First we show that $\varphi_s \in \mathcal{F}_{F_{1_\chi}}\cap \mathcal{F}_{F_{2_\chi}}$ and that $L_{F_{j_\chi}}(\varphi_s)$ is holomorphic as a function of $s\in\mathbb{C}$ (this will be to justify Mellin inversion). We have
\begin{equation}
\begin{split}
\mathcal{L}|\varphi_{s_{k+1}}|\left(\frac{2\pi m}{M}\right)
&=\int_0^\infty|y^{s+k-1}||\varphi(y)|e^{-2\pi y m/M}dy\\
&=\int_{c_1}^{c_2}|y^{s+k-1}||\varphi(y)|e^{-2\pi y m/M}dy\\
&<<_\varphi \text{max}\left\{c_1^{\text{Re}(s)+k-1},c_2^{\text{Re}(s)+k-1}\right\} e^{-2\pi c_1 m/M}.
\end{split}
\end{equation}
Now an application of the Weierstrass $M$-test shows that $L_{F_{j_\chi}}(\varphi_s)$ converges absolutely (so $\varphi_s \in \mathcal{F}_{F_{1_\chi}}\cap \mathcal{F}_{F_{2_\chi}}$) and uniformly on compact subsets, and therefore the uniform limit theorem tells us this limit is holomorphic. With this, we can apply the following Mellin inversions:
\begin{align}
 F_{j_{\chi}} (iy) \varphi (y) &= \dfrac{1}{2\pi i} \int_{(\sigma)} L_{F_{j_{\chi}}}(\varphi_{r+s-t}) y^{-r-s+t} dt, \label{ConverseFj}
 \\
\partial_r(F_{j_{\chi}}) (iy) \varphi (y) &= \dfrac{1}{2\pi i} \int_{(\sigma)} L_{\partial_r F_{j_{\chi}}}(\varphi_{r+s-t}) y^{-r-s+t} dt, \label{ConversedrFj}
 \\
\overline{\partial}_s(F_{j_{\chi}}) (iy) \varphi (y) &= \dfrac{1}{2\pi i} \int_{(\sigma)} L_{\bar\partial_s F_{j_{\chi}}}(\varphi_{r+s-t}) y^{-r-s+t} dt.\label{ConversedsFj}
\end{align}
Combining equations \eqref{ConverseFj} and \eqref{Converse1a} then gives
\begin{equation}\label{ConverseF1F2}
    F_{1_{\chi}} (iy) \varphi (y) = \dfrac{1}{2\pi i}  \frac{\chi(-N)\psi(D)}{i^{s-r} N^{\frac{r+s-2}{2}}} 
 \int_{(\sigma)} L_{F_{2_{\bar\chi}}}(\varphi_{r+s-t}|_{2-r-s} W_N) y^{-r-s+t} dt.
\end{equation}

The next step is to rewrite the above equation into one that is of more use to us. We first note that $L_{F_{2_{\bar\chi}}}(\varphi_{r+s-t}|_{2-r-s} W_N) = \int^\infty_0 F_{2_{\bar\chi}}(iy)(\varphi_{r+s-t}|_{2-r-s} W_N)(y)dy$ and 
\begin{align*}
    (\varphi_{r+s-t}|_{2-r-s} W_N)(y) & = (Ny)^{r+s-2}\varphi_{r+s-t}(1/Ny) \\[2mm]
    & = (Ny)^{r+s-2}\varphi(1/Ny)(Ny)^{1+t-r-s}
    = (Ny)^{t-1}\varphi(1/Ny).
\end{align*}
Therefore, we have 
\begin{equation*}
    L_{F_{2_{\bar\chi}}}(\varphi_{r+s-t}|_{2-r-s} W_N) = \int^\infty_0 F_{2_{\bar\chi}}(iy) (Ny)^{t-1}\varphi(1/Ny) dy
\end{equation*}
and, by changing the variable $y \to 1/Ny$, this becomes
\begin{equation*}
    L_{F_{2_{\bar\chi}}}(\varphi_{r+s-t}|_{2-r-s} W_N) = \dfrac{1}{N} \int^\infty_0 F_{2_{\bar\chi}}(i/Ny) y^{-t-1}\varphi(y) dy.
\end{equation*}
By Mellin inversion, we then have
\begin{equation*}
    \dfrac{1}{2\pi i} \int_{(\sigma)} L_{F_{2_{\bar\chi}}}(\varphi_{r+s-t}|_{2-r-s} W_N) \, y^t dt = \dfrac{1}{N} F_{2_{\bar\chi}}(i/Ny) \varphi(y).
\end{equation*}
and conclude that equation \eqref{ConverseF1F2} can be written as
\begin{equation}\label{transformlawconv}
    F_{1_{\chi}}(iy) \varphi(y) = \frac{\chi(-N)\psi(D)}{i^{s-r} N^{\frac{r+s}{2}}}  y^{-r-s} F_{2_{\bar\chi}}(i/Ny) \varphi(y).
\end{equation}

Since (\ref{transformlawconv}) holds for all $\varphi \in S_c(\mathbb{R}_+)$, for any $y>0$ we can always choose a $\varphi$ such that $\varphi(y)\neq 0$, and therefore (\ref{transformlawconv}) implies
\begin{equation}\label{ConverseF1=F2}
    F_{1_{\chi}}(iy)  = \frac{\chi(-N)\psi(D)}{i^{s-r} N^{\frac{r+s}{2}}}  y^{-r-s} F_{2_{\bar\chi}}(i/Ny) .
\end{equation}
Following the same procedure as above we also obtain the following equations for the operators $\partial$ and $\overline\partial$:
\begin{align*}
    \partial_r (F_{1_{\chi}})(iy)  & = -\frac{\chi(-N)\psi(D)}{i^{s-r} N^{\frac{r+s}{2}}}  y^{-r-s} \partial_r(F_{2_{\bar\chi}})(i/Ny), 
    \\
    \overline{\partial}_s (F_{1_{\chi}})(iy) \,  & = -\frac{\chi(-N)\psi(D)}{i^{s-r} N^{\frac{r+s}{2}}}  y^{-r-s} \overline{\partial}_s (F_{2_{\bar\chi}})(i/Ny) .
\end{align*}
We then define
\begin{equation*}
    I_\chi(z) \coloneqq F_{1_\chi}(z) - \chi(-N)\psi(D)(F_{2_{\bar\chi}}||_{r,s} W^{-1}_N)(z)
\end{equation*}
and claim that $I_\chi(iy) = 0$. To see that this is true we simply note that
\begin{align*}
    (F_{2_{\bar\chi}}||_{r,s} W^{-1}_N)(iy) & = (- \sqrt{N} i y)^{-r}( \sqrt{N} i y)^{-s}F_{2_{\bar\chi}}(-1/iyN) \\
    & = N^{\frac{-r-s}{2}}y^{-r-s} (-i)^{-r}i^{-s} F_{2_{\bar\chi}}(i/yN) = \dfrac{y^{-r-s}}{i^{s-r}N^{\frac{r+s}{2}}} F_{2_{\bar\chi}}(i/yN)
\end{align*}
and use equation \eqref{ConverseF1=F2}.

Similarly, we get $\partial_r I_{\chi}(iy)=0$ and $\overline{\partial}_s I_\chi(iy)=0$. This, in turn, implies that
\begin{equation*}
    \frac{\partial}{\partial z} I_\chi (iy) = \frac{\partial}{\partial\bar z} I_\chi (iy) = 0
\end{equation*}
and 
\begin{equation*}
    \frac{\partial}{\partial x} I_\chi (iy) 
    = \left( \frac{\partial}{\partial z} +  \frac{\partial}{\partial \bar z}\right)  I_\chi (iy) = 0.
\end{equation*}
To summarise, we have $I_\chi (iy) = \frac{\partial}{\partial x} I_\chi (iy) = 0$, which allows us to use Lemma \ref{BumLem} to conclude that $I_\chi \equiv 0$ and
\begin{equation*}
     F_{1_\chi}(z) = \chi(-N)\psi(D)(F_{2_{\bar\chi}}||_{r,s} W^{-1}_N)(z).
\end{equation*}
This is essentially equation (5.28) from Theorem 5.1 of \cite{diamantisLseries}, but with $|_k$ replaced by $||_{r,s}$ (in our notation $|_k = ||_{k,0} $). The same reasoning used below equation (5.28) in the proof of Theorem 5.1 of \cite{diamantisLseries} shows that $F_{1}$ obeys the transformation law of equation \eqref{transform}. We see that $F_{1}$ obeys equation \eqref{fourier} directly from \eqref{Conversecusp}.

To show that $F_1$ satisfies condition (iii) of the definition of $\mathcal M_{r, s}^!(\Gamma(N), \chi)$, we first note the following analogue of Lemma 4.33 of \cite{miyakebook}:
\begin{lemma}\label{miyaketype} Let $f:\mathbb{H}\rightarrow\mathbb{C}$ satisfy the assumptions \emph{(\ref{fexpform})}.  Then, for some $A, B>0,$
\begin{equation}\label{bounds} f(x+iy)=O(e^{Ay}) \, \, \text{as $y \to \infty$, and} \, \, =O(e^{B/y}) \, \, \text{as $y \to 0$}
\end{equation}
uniformly in $x.$
\end{lemma}
\begin{proof}
    This follows from the form of the assumed expansion of $f$.
\end{proof}
Then for $\gamma=\begin{psmallmatrix}
* &  * \\ c &  d  \end{psmallmatrix} \in SL_2(\mathbb Z)$ with $c=0$, condition $(iii)$ of Definition \ref{ramfdef} follows directly from the first part of \eqref{bounds}.

If $c \neq 0,$ we first note that by, say, Cor. 5.1.14 of \cite{cohen1}, $f||_{r, s}\gamma$ is periodic with some period $m_0$, so we can restrict to $z=x+iy$ with $x \in [0, m_0].$ Then, as $y \to \infty$, we have that  $\text{Im}(\gamma z)=\text{Im}(z)/((cx+d)^2+(cy)^2) \to 0$ and thus $$(f||_{r,s} \gamma)(z)=(cz+d)^{-r}(c\bar{z}+d)^{-s}f(\gamma z)=O(e^{B((cx+d)^2+(cy)^2)/y})=O(e^{B' y})$$ for some $B'>0$ as required for $(iii)$ of Definition \ref{ramfdef} to hold. We therefore conclude that $F_1 \in \mathcal{M}^!_{r,s}(\Gamma_0(N), \psi)$.
\end{proof}
\subsection{Functional equation and converse theorem for general \emph{L}-series}
As mentioned in the introduction, we can define an $L$-series for all real-analytic modular forms by introducing an auxiliary variable.
\begin{definition}
Let $F$ have an expansion of the form
\begin{equation}\label{expformu}
F(z)=\sum_{|j| \le N_0}y^j\left ( \sum_{m, n \ge -N_0'} a_{m, n}^{(j)} \, q^{m/M} \overline{q}^{n/M}\right )\! , 
\end{equation}
for some $N_0,N_0'\in\mathbb{N}$, $a_{m,n}^{(j)}\in\mathbb{C}$. Then we let $\widetilde{\mathcal{F}}_F\subset\mathcal{F}_F$ denote the space of piecewise continuous functions $\mathbb{R}\rightarrow\mathbb{C}$ such that the following series converges for all $u\in\mathbb{R}$, $|j|\le N_0$:

\begin{equation}\label{Ffconvu}
\sum_{\substack{m,n\in\mathbb{Z}\\m,n\geq -N_0'}}|a_{m,n}^{(j)}|\mathcal{L}|\varphi_{j+1}|\left(\frac{2\pi(n+m)}{M\sqrt{u^2+1}}\right)<\infty.
 \end{equation}
 And the $L$-series associated to $F$ is defined to be the map  $L_F: \mathcal{F}_F \times\mathbb{R} \to \mathbb{C}$ given by
    \begin{align*}
        &L_F(\varphi;u)=\sum_{|j|\leq N_0}\frac{1}{\sqrt{u^2+1}^{j+1}}\sum_{\substack{m,n\in\mathbb{Z}\\m,n\geq -N_0'}}a_{m,n}^{(j)} \,\mathcal{L}\varphi_{j+1}\left(\frac{2\pi(n+m+iu(n-m))}{M\sqrt{u^2+1}}\right) \! .
    \end{align*}
\end{definition}
We note that if $F$ has an expansion of the form given by (\ref{F/M}) then $L_F(\varphi;0)=L_F(\varphi)$. The analogue of Lemma \ref{integralformlemma} is as follows:
\begin{lemma}
Let $F$ have an expansion of the form (\ref{expformu}) and $\varphi\in\widetilde{\mathcal{F}}_F$. Then
\begin{equation}
L_F(\varphi;u)=\int_0^\infty F((i+u)y)\varphi(y\sqrt{u^2+1}) dy.
\end{equation}
\end{lemma}
We can use this in much the same way to deduce the following functional equation:
\begin{theorem}\label{functionalequ}
Let $r,s \in \mathbb{Z}$, $N,D \in \mathbb{N}$ with $(D,N)=1$, and let $\psi$ and $\chi$ be Dirichlet characters modulo $N$ and $D$, respectively. Furthermore, let $F\in\mathcal{M}^!_{r,s}(\Gamma_0(N), \psi)$ and have an expansion of the form
\begin{equation}
 F(z) = \sum_{|j| \le N_0}y^j\left ( \sum_{m, n \ge -N_0'} a_{m, n}^{j} \, e^{2\pi i m z}e^{-2\pi i n\overline{z}}\right ) \!,
\end{equation}
with coefficients not growing too rapidly e.g.
\begin{equation}\label{coeffgrowthu}
\sum_{m,n>0}|a_{m,n}^{(j)}|e^{-(n+m)c}
\end{equation}
converges for all $c>0$, $|j|\leq N_0$.
Suppose a similar expansion and coefficient bounds for
\begin{align*}
G \coloneqq F||_{r,s} W_N.
\end{align*}
Finally, we define the test function space
\begin{align*}
\widetilde{\mathcal{F}}_{F,G} = \bigcap_{\chi\mod D}\left\{\varphi\in\widetilde{\mathcal{F}}_{F_\chi}:\varphi|_{2-r-s} W_N\in\widetilde{\mathcal{F}}_{G_{\bar\chi}}\right\}.
\end{align*}
Then $\widetilde{\mathcal{F}}_{F,G}$ is non-empty and, for $\varphi \in \widetilde{\mathcal{F}}_{F,G}$ and $u\in\mathbb{R}$, we have the following functional equation:
\begin{align}
     L_{F_\chi}(\varphi ; u) =  \frac{(i-u)^r(-i-u)^s\chi(-N)\psi(D)}{(u^2+1)^{(r+s)/2} N^{\frac{r+s-2}{2}}} L_{G_{\bar\chi}}(\varphi|_{2-r-s} W_N; -u) \label{LFX1bu} .
\end{align}

\end{theorem}

Again, when $u=0$ and $F$ has an expansion like (\ref{F/M}), this simply reduces to Theorem \ref{functionaleq}. The proof goes through much the same as Theorem \ref{functionaleq}, except instead of the substitution $y\rightarrow 1/(Ny)$, we use $y \to 1/(Ny(u^2+1))$.

For the converse theorem, previously we used the analogue of Bump's Lemma, Lemma \ref{BumLem}. In this case, the auxiliary variable $u$ plays the role of that lemma, since now our functional equations essentially define $F$ on all of $\mathbb{H}$, rather than the imaginary axis only (recall equation \eqref{ConverseF1=F2}).

\begin{theorem}\label{conversethu}
    Let $N \geq 1$ be an integer, $\psi$ be a Dirichlet character modulo $N$. For some $N_0,N_0' \in \mathbb{N}$, let $a^{(j)}_{m,n}$ and $b^{(j)}_{m,n}$, with, $j \in \{-N_0, -N_0+1,...,N_0 \}$ and $n,m\geq -N_0'$, be sequences of complex numbers with condition $(\ref{coeffgrowthu})$.
    We can define the smooth functions $F_1:\mathbb{H} \to \mathbb{C}$,
        \begin{align*}
         F_1(z) =\sum_{|j| \le N_0}y^j\left ( \sum_{m, n \ge -N_0'} a_{m, n}^{(j)} \, e^{2 \pi i m z} e^{-2 \pi i n \bar{z}}\right )
    \end{align*}
    and $F_2$
    \begin{align}
         F_2(z) & =\sum_{|j| \le N_0}y^j\left ( \sum_{m, n \ge -N_0'} b_{m, n}^{(j)} \, e^{2 \pi i m z} e^{-2 \pi i n \bar{z}}\right ). \label{convfour}
    \end{align}
    Now, for all $D \in \{1,2,...,N^2-1 \}$ with gcd$(D,N)=1$, let $\chi$ be a Dirichlet character mod $D$. If, for any $\varphi\in S_c(\mathbb{R_+})$, $D$ and $\chi$, the following equation holds:
\begin{equation}
     L_{F_{1_\chi}}(\varphi ; u) =  \frac{(i-u)^r(-i-u)^s\chi(-N)\psi(D)}{(u^2+1)^{(r+s)/2} N^{\frac{r+s-2}{2}}} L_{F_{2_{\bar\chi}}}(\varphi|_{2-r-s} W_N; -u) \label{Conv1},
\end{equation}
for all $u \in \mathbb{R}$, then $(F_{1_{\chi}}||_{r,s}W_N) (z) = \chi(-N)\psi(D)   F_{2_{\bar\chi}}(z)$ and $F_1\in\mathcal{M}^!_{r,s}(\Gamma_0(N),\psi)$.
\end{theorem}
\begin{proof}[Proof sketch]
We start with the equation
\begin{align*}
        L_{F_{1_\chi}}(\varphi_t ; u) & = \int^{\infty}_0 F_{1_\chi}((i+u)y)\varphi(y\sqrt{u^2+1})(y\sqrt{u^2+1})^{t-1}dy
        \\ & = \dfrac{1}{\sqrt{u^2+1}} \int^{\infty}_0 F_{1_\chi}\left(\frac{(i+u)y}{\sqrt{u^2+1}}\right)\varphi(y)y^{t-1}dy.
    \end{align*}
    Mellin inversion gives
    \begin{equation*}
        \dfrac{1}{\sqrt{u^2+1}}  F_{1_\chi}\left(\frac{(i+u)y}{\sqrt{u^2+1}}\right)\varphi(y) = \frac{1}{2\pi i}\int_{(\sigma)} L_{F_{1_\chi}}(\varphi_t ; u) y^{-t} dt.
    \end{equation*}
    Setting $y= 1/(Ny\sqrt{u^2+1})$ gives
    \begin{equation*}
        \dfrac{1}{\sqrt{u^2+1}}  F_{1_\chi}\left(\frac{-1}{Ny(i-u)}\right)\varphi\left(\frac{1}{Ny\sqrt{u^2+1
    }}\right) = \frac{1}{2\pi i}\int_{(\sigma)} L_{F_{1_\chi}}(\varphi_t ; u) (Ny\sqrt{u^2+1})^{t} dt.
    \end{equation*}
    Moving the line of integration from $(\sigma)$ to $(r+s-\sigma)$ and changing variables $t \to r+s-t$, we get
\begin{equation*}
    \dfrac{1}{\sqrt{u^2+1}}  F_{1_\chi}\left(\frac{-1}{Ny(i-u)}\right)\varphi\left(\frac{1}{Ny\sqrt{u^2+1
    }}\right) = \frac{1}{2\pi i}\int_{(\sigma)} L_{F_{1_\chi}}(\varphi_{r+s-t} ; u) (Ny\sqrt{u^2+1})^{r+s-t} dt.
\end{equation*}
Using equation \eqref{Conv1}, we have
    \begin{align}
         \label{conv2}
   & \dfrac{1}{\sqrt{u^2+1}}   F_{1_\chi}\left(\frac{-1}{Ny(i-u)}\right)\varphi\left(\frac{1}{Ny\sqrt{u^2+1
    }}\right)  \\[2mm] &\quad =  \frac{(i-u)^r(-i-u)^s\chi(-N)\psi(D)}{(u^2+1)^{(r+s)/2} N^{\frac{r+s-2}{2}}} 
 \frac{1}{2\pi i}\int_{(\sigma)} L_{F_{2_{\bar\chi}}}(\varphi_{r+s-t}|_{2-r-s} W_N; -u)  (Ny\sqrt{u^2+1})^{r+s-t} dt. \nonumber
    \end{align}
Now expand $L_{F_{2_{\bar\chi}}}(\varphi_{r+s-t}|_{2-r-s} W_N; -u) = \int^{\infty}_0 F_{2_{\bar\chi}}((i-u)y) \,(\varphi_{r+s-t}|_{2-r-s} W_N)(y\sqrt{u^2+1}) dy$:
    \begin{align*}
        (\varphi_{r+s-t}|_{2-r-s} W_N)(y\sqrt{u^2+1})  & = (Ny\sqrt{u^2+1})^{r+s-2}\varphi_{r+s-t}\left(\frac{1}{Ny\sqrt{u^2+1}}\right)
        \\ & = (Ny\sqrt{u^2+1})^{r+s-2}(Ny\sqrt{u^2+1})^{t+1-r-s} \varphi\left(\frac{1}{Ny\sqrt{u^2+1}}\right)
        \\ & = (Ny\sqrt{u^2+1})^{t-1} \varphi\left(\frac{1}{Ny\sqrt{u^2+1}}\right),
    \end{align*}
    and therefore, we have
    \begin{align*}
L_{F_{2_{\bar\chi}}}(\varphi_{r+s-t}|_{2-r-s} W_N; -u) = \int^{\infty}_0 F_{2_{\bar\chi}}((i-u)y) (Ny\sqrt{u^2+1})^{t-1} \varphi\left(\frac{1}{Ny\sqrt{u^2+1}}\right) dy
\\ = \dfrac{1}{N\sqrt{u^2+1}} \int^{\infty}_0 F_{2_{\bar\chi}}\left(\frac{(i-u)y}{N\sqrt{u^2+1}}\right) y^{t-1} \varphi\left(\frac{1}{y}\right) dy.
    \end{align*}
    Mellin inversion gives:
    \begin{equation*}
        \dfrac{1}{2\pi i }\int_{(\sigma)} L_{F_{2_{\bar\chi}}}(\varphi_{r+s-t}|_{2-r-s} W_N; -u)y^{-t} dt = \dfrac{1}{N\sqrt{u^2+1}}  F_{2_{\bar\chi}}\left(\frac{(i-u)y}{N\sqrt{u^2+1}}\right) \varphi\left(\frac{1}{y}\right) 
    \end{equation*}
    and then setting $y = Ny\sqrt{u^2+1}$ gives
    \begin{align*}
        \dfrac{1}{2\pi i }\int_{(\sigma)} L_{F_{2_{\bar\chi}}}(\varphi_{r+s-t}|_{2-r-s} W_N; -u)&(Ny\sqrt{u^2+1})^{-t} dt \\ & = \dfrac{1}{N\sqrt{u^2+1}}  F_{2_{\bar\chi}}((i-u)y) \varphi\left(\frac{1}{Ny\sqrt{u^2+1}}\right) \!.
    \end{align*}
    Now, equation \eqref{conv2} becomes
\begin{align}
    \nonumber
   &  \dfrac{1}{\sqrt{u^2+1}}   F_{1_\chi}\left(\frac{-1}{Ny(i-u)}\right)\varphi\left(\frac{1}{Ny\sqrt{u^2+1
    }}\right)  \\[2mm] &\ \, \ =  \frac{(i-u)^r(-i-u)^s\chi(-N)\psi(D)}{(u^2+1)^{(r+s)/2} N^{\frac{r+s-2}{2}}} 
(Ny\sqrt{u^2+1})^{r+s}  \dfrac{1}{N\sqrt{u^2+1}}  F_{2_{\bar\chi}}((i-u)y) \varphi\left(\frac{1}{Ny\sqrt{u^2+1}}\right) \!. \nonumber
    \end{align}
    Hence
    \begin{align*}
         F_{1_\chi}\left(\frac{-1}{Ny(i-u)}\right) & =  \frac{(i-u)^r(-i-u)^s\chi(-N)\psi(D)}{ N^{\frac{r+s-2}{2}}} 
N^{r+s-1}y^{r+s}   F_{2_{\bar\chi}}((i-u)y)
\\ & = ((i-u)y)^r((-i-u)y)^s N^{\frac{r+s}{2}} \chi(-N)\psi(D)   F_{2_{\bar\chi}}((i-u)y).
    \end{align*}
    Since this holds for all $u \in \mathbb{R}$, this holds for $u=-x/y$. Using this substitution, the above equation becomes
    \begin{equation*}
        F_{1_\chi}\left(\frac{-1}{N(x+iy)}\right)  = (x+iy)^r(x-iy)^s N^{\frac{r+s}{2}} \chi(-N)\psi(D)   F_{2_{\bar\chi}}(x+iy)
    \end{equation*}
    and hence
    \begin{equation*}
        F_{1_\chi}\left(\frac{-1}{Nz}\right) = z^r\bar{z}^s N^{\frac{r+s}{2}} \chi(-N)\psi(D)   F_{2_{\bar\chi}}(z).
    \end{equation*}
    We then conclude that
    \begin{align*}
        \chi(-N)\psi(D)   F_{2_{\bar\chi}}(z) & =  z^{-r}\bar{z}^{-s} N^{\frac{-r-s}{2}}  F_{1_\chi}\left(\frac{-1}{Nz}\right)
        \\ & = (F_{1_\chi}||_{r,s}W_N)(z),
    \end{align*}
    as required. We then conclude in the same way as in Theorem $\ref{converseth}$ that $F_1$ satisfies $(i)$ of Definition \ref{ramfdef} with weights $(r,s)$. We can use a result similar to Lemma \ref{miyaketype} to verify that $F_1$ has sufficient growth at the cusps for condition $(iii)$.
\end{proof}

\bibliography{paperrefs} 
\bibliographystyle{ieeetr}

\end{document}